\RequirePackage{fix-cm}
\RequirePackage{amsmath}
\documentclass{amsart}
\usepackage{pinlabel}
\usepackage[all]{xy}
\usepackage[mathscr]{euscript}
\usepackage{epsfig,amscd,amssymb,amsmath,amsfonts}
\usepackage{color}
\usepackage{tikz}
\usepackage{amsthm}
\newtheorem{theorem}{Theorem}[section]

\newtheorem{proposition}{Proposition}[section]
\newtheorem{definition}[theorem]{Definition}

\newtheorem{example}[theorem]{Example}

\newtheorem{remark}[theorem]{Remark}

\numberwithin{equation}{section}

\newcommand{\kmod}{k\textrm{-mod}}
\begin{document}

\title{Higher Order Hochschild (Co)homology of Noncommutative Algebras}

\author{Bruce R. Corrigan-Salter}        

%\address{Department of Mathematics \at Wayne State University, Detroit, MI 48202, USA \\
%\email{brcs@wayne.edu}}          

%\author{Bruce R. Corrigan-Salter}
\address{Department of Mathematics,
Wayne State University,
Detroit, MI 48202, USA}
\email{brcs@wayne.edu}

\begin{abstract} 
Hochschild (co)homology and Pirashvili's higher order Hochschild (co)homology are useful tools for a variety of applications including deformations of algebras.  When working with higher order Hochschild (co)homology, we can consider the (co)homology of any commutative algebra with symmetric coefficient bimodules, however traditional Hochschild (co)homology is able to be computed for any associative algebra with not necessarily symmetric coefficient bimodules.  In a previous paper, the author generalized higher order Hochschild cohomology for multimodule coefficients (which need not be symmetric).  In the current paper, we continue to generalize higher order Hochschild (co)homology to work with associative algebras which need not be commutative and in particular, show that simplicial sets admit such a generalization if and only if they are one dimensional.

\end{abstract}

\maketitle

%\section*{This is an unnumbered first-level section head}

%
%%%%%%%%%%%%%%%%%%%%%%%%%%%%%%%%%%%%%%%%%%%%%%%%%%%

%%%%%%%%%%%%%%%%%%%%%%%%%%%%%%%%%%%%%%%%%%%%%%%%%%%%
\section*{Introduction}
Hochschild (co)homology was first introduced in 1945 by Hochschild in \cite{MR0011076}.  Since then, mathematicians have found Hochschild (co)homology to be incredibly useful for a variety of applications.  In the past couple decades, energy has been placed into generalizing Hochschild's construction.  One generalization, higher order Hochschild homology was introduced by Pirashvili in \cite{MR1755114}.  This construction assigns a chain complex to any simplicial set, provided the algebra is commutative and the bimodule is symmetric.  More recently, the author has worked to generalize Pirashvili's construction to work with multimodules, modules who have more than one action  (see \cite{MR3338542}).  It is the aim of the present paper to further generalize this construction to noncommutative algebras.  In particular, we aim to determine a list of all simplicial sets whose Hochschild (co)homology is defined even when working over noncommutative algebras.  We do so by finding the ``maximal'' algebraic structure allowed by each simplicial set $X_\bullet$.  The main result found here is that the Hochschild (co)homology of $X_\bullet$ is defined over a noncommutative algebra $A$ if and only if $X_\bullet$ is a one dimensional simplicial set.  

Given a field, $k$, $k$-algebra, $A$ and an $A$-bimodule, $M,$ we can associate a chain complex $C_{\bullet}(A,M),$ whose homology was introduced by Hochschild in \cite{MR0011076} and is referred to as the Hochschild homology of $A$ with coefficients in $M.$
To define $C_{\bullet}(A,M),$ let               
$$C_n(A,M)\colon = M\otimes A^{\otimes n}$$ and differentials be given by $$\delta_n = \sum_{i=0}^n (-1)^id_i$$ where $d_i \colon M\otimes A^{\otimes n} \rightarrow M\otimes A^{\otimes n-1}$ are defined as follows:
$$d_i(m\otimes a_1 \otimes a_2 \otimes \cdots \otimes a_n)= \left\{
                \begin{array}{ll}
                  ma_1\otimes a_2\otimes \cdots \otimes a_n & i=0\\
                  m\otimes a_1\otimes \cdots \otimes a_i a_{i+1} \otimes \cdots \otimes a_n & 1\leq i \leq n-1\\
                  a_nm\otimes a_1 \otimes \cdots \otimes a_{n-1} & i=n
                \end{array}
              \right.$$
              It can be observed that the construction above describes a simplicial $k$-module associated to a simplicial model of $S^1$.  Furthermore, Pirashvili illustrated in \cite{MR1755114} that we can start with any arbitrary finite simplicial set and construct a simplicial $k$-module whose associated homology (of the complex given by alternating face maps) gives a generalization of Hochschild homology, known as higher order Hochschild homology.  For this construction we start with the assumptions that $A$ is a commutative $k$-algebra and $M$ is a symmetric $A$-bimodule.  These assumptions are necessary as we will see from Theorem \ref{MAIN}.  For example $S^2$ (with the minimal simplicial decomposition) requires that $A$ be commutative (see Proposition \ref{SPHERE}). For higher order Hochschild homology, consider the Loday functor $L(A,M)$ from the category of finite pointed sets, $\Gamma$ to the category of $k$-modules, $\kmod$ (see \cite[6.4.4]{MR1217970} and \cite{MR1755114}), given by      $$L(A,M)\colon \Gamma \rightarrow \kmod$$ $$m_+ \rightarrow M\otimes \bigotimes_{\{i\in m_+ |i\neq 0\}} A$$ for objects $m_+=\{0,1,\cdots,m\}\in \Gamma.$ (where $0$ is the fixed element).  For morphisms $\varphi\colon m_+ \rightarrow n_+ \in \Gamma$ let $$L(A,M)\varphi(m\otimes a_1\otimes \cdot \otimes a_m)=(b_0m\otimes \cdots \otimes b_n)$$ where $$b_i=\prod_{\{j\in m_+|j\neq 0, \varphi(j)=i \}}a_j.$$
              
Similarly we have a functor \cite{MR2383113} $$\mathcal{H}(A,M) \colon \Gamma \rightarrow \kmod$$ $$m_+ \rightarrow \hom (\bigotimes_{\{i\in m_+ |i\neq 0\}} A,M)$$ where for a map $\varphi \colon m_+ \rightarrow n_+$ and map $f\colon \bigotimes_{\{i\in n_+ |i\neq 0\}} A\rightarrow M$ we have 
              $$\mathcal{H}(A,M)\varphi (f)( a_1 \otimes \cdots \otimes a_m)=b_0f(b_1\otimes \cdots \otimes b_n)$$ where $$b_i=\prod_{\{j\in m_+ |j\neq 0, \varphi(j)=i\}}a_j.$$
              
It was realized by Pirashvili in \cite[3.1]{MR1755114} that for any finite simplicial set $X_\bullet ,$ we can consider the chain complex (after defining differentials to be the sums of alternating face maps) $$\Delta^{op}\xrightarrow{X_\bullet}\Gamma\xrightarrow{L(A,M)}\kmod$$ or similarly the cochain complex (see \cite{MR2383113}) $$\Delta^{op}\xrightarrow{X_\bullet}\Gamma\xrightarrow{\mathcal{H}(A,M)} \kmod.$$
             The resulting homologies are referred to as higher order Hochschild homology and higher order Hochschild cohomology respectively.                            
              
To see how this generalizes traditional Hochschild (co)homology we consider the minimal simplicial decomposition of the pointed simplicial set $S^1_\bullet \colon \Delta^{op}\rightarrow \Gamma$ (with one non-degenerate $1$-simplex).  We have a simplicial $k$-module $$\Delta^{op}\xrightarrow{S_\bullet^1} \Gamma \xrightarrow{L(A,M)} \kmod$$ which gives the Hochschild chain complex.  The associated cochain complex $\Delta^{op}\xrightarrow{S_\bullet^1}\Gamma \xrightarrow{\mathcal{H}(A,M)} \kmod$ is the Hochschild cochain complex. The resulting homologies are known as Hochschild homology of $A$ with coefficients in $M$ and Hochschild cohomology of $A$ with coefficients in $M$ respectively.

             Hochschild (co)homology and higher order Hochschild (co)homology have been shown to be incredibly useful tools for a variety of concepts. Hochschild cohomology and higher order Hochschild cohomology have been used to study deformations of algebras and modules (for example see \cite{MR0161898}, \cite{MR0171807}, \cite{MR981619}, \cite{MR551624}, \cite{MR1217970},
             \cite{MR3465889} and \cite{MR2153227}).  The ability to use higher order Hochschild cohomology for additional deformations is a primary inspiration for this work.  One has to wonder if deformations of noncommutative algebras would be possible in the higher order setting and this paper seeks to provide those studying deformation theory with an answer about the types of algebras allowed.  In addition to being useful in studying deformation theory, in 1976 Dennis developed a map from K-theory to Hochschild homology which was benefited by the introduction of topological Hochschild homology in \cite{bokstedt} by B{\"o}kstedt in 1985 as an answer to Goodwillie's conjecture (see \cite{MR3013261}).  In addition, the algebraic structures of Hochschild (co)homology, which include the Deligne conjecture, have been of great interest (see a treatment by Tradler and Zeinalian for example \cite{MR2184812} and \cite{MR2353864}).

             One thing that can be noticed when considering the history of Hochschild cohomology is that traditional Hochschild (co)homology is of any associative $k$-algebra, $A$ with coefficients in any $A$-bimodule, $M,$ however higher order Hochschild (co)homology was restricted to commutative $k$-algebras with coefficients in symmetric $A$-bimodules.  In \cite{MR3338542} the author generalizes the higher order Hochschild construction to have coefficients in not necessarily symmetric multimodules (see Definition \ref{multimodule}).  The modules able to be used depend heavily on the simplicial sets that the (co)chain complexes are built over.  In this paper we aim to generalize higher order Hochschild (co)homology to work with not necessarily commutative algebras and not necessarily symmetric multimodules.  In particular,  we will show that simplicial sets allow such a generalization to noncommutative algebras if and only if they are one dimensional.  We will do so by demonstrating the best case scenario for each simplicial set.  Our construction for all arbitrary simplicial sets is in line with Pirashvili and Richter's construction in \cite{MR1899698} (see Subsection \ref{PR}) where they gave a description of a simplicial noncommutative circle as a functor from the simplicial category, $\Delta^{op}$ to the category of finite noncommutative sets, which allowed them to construct Hochschild (co)homology (as well as other homologies of functors) with not necessarily commutative algebras and not necessarily symmetric bimodules.  

While this paper is focused on simplicial sets and Pirashvili's generalization, it is worth mentioning that there are additional equivalent generalizations.  In particular, Lurie's topological chiral homology (see \cite{MR2555928}), which is equivalent to Francis's factorization homology (see \cite{MR3040746}) is a generalization of Hochschild homology, which associates to a n-framed manifold $N$ and little n-cubes algebra $A,$ a chain complex.  When $N$ is $S^1,$ and $A$ is an associative algebra, it has been shown by both Lurie and Francis that the chain complex is equivalent to the traditional Hochschild chain complex.  Furthermore, Ginot, Tradler and Zeinalian have shown in \cite{MR3173402} that for a commutative algebra $A$ and n-framed manifold $N,$ factorization/topological chiral homology is equivalent to the higher order Hochschild chain complex.  For direct constructions of these chain complexes, see \cite{Markarian}.  In \cite{MR3431668}, Ayala and Francis have even shown that there is an equivalence between the category of little n-cubes algebras and homology theories of n-framed manifolds via factorization homology.  

 \subsection{Organization of Paper}
In Section \ref{BACK}, we provide background knowledge about the authors previous work on coefficient modules.  In Section \ref{notnec} we describe how to construct the higher order Hochschild cohomology cochain complex for algebras which are not necessarily commutative.  We do so by describing the necessary characteristics a simplicial set must have in order to work with noncommutative algebras, which amounts to considering orderings on fibers of face maps.  Our main Theorem is provided in Section \ref{SSTHATWORK}.  It determines that the only simplicial sets that can work with noncommutative algebras are one dimensional. In Section \ref{second} we note that the construction in this paper can be extended to higher order Hochschild cohomology of pairs of simplicial sets from \cite{MR3566502}.

\section{Background}
\label{BACK}
In this paper we fix a field $k$ and denote $\otimes_k$ by $\otimes$.  We assume that the reader is familiar with simplicial sets and has some familiarity with Hochschild (co)homology.  We hope that the introduction along with an extra explanation below provide a sufficient description of Hochschild (co)homology.

              \subsection{Higher Order Hochschild Cohomology}
If we consider the definition of higher order Hochschild homology we notice that the action of $A$ on $M$ tells us what happens in the $M$ tensor factor of the codomain.  The reason $M$ is not allowed to be a nonsymmetric bimodule is because it would be difficult in general to determine which of the two actions of $A$ on $M$ to choose for each face map.  In addition, to get a simplicial $k$-module it is easier to satisfy the composition laws by choosing one action (which is forced to be both a left and right action since $A$ is commutative).  In \cite{MR3338542} the author constructs a generalized version of higher order Hochschild cohomology which allows multimodule  coefficients.  In order to allow more actions for the coefficient modules and still get a cosimplicial $k$-module, we need to make action identifications.  

Before proceeding with the main Theorem from \cite{MR3338542}, we provide a definition of multimodules.  

\begin{definition}
\label{multimodule}
Let $A$ be a $k$-algebra.  A \emph{left} $n$-\emph{multimodule} is an abelian group $M$ with $n$ distinct left $A$-module structures $\iota_i$ which commute in the sense that $\iota_i(a_i)\iota_j(a_j)m=\iota_j(a_j)\iota_i(a_i)m$ for all $a_i, a_j \in A, m \in M, 1\leq i \neq j \leq n.$  Similarly, a \emph{right} $n$-\emph{multimodule} is an abelian group $M$ with $n$ distinct right $A$-module structures $\rho_i$ which commute in the sense that $m \rho_i(a_i)\rho_j(a_j)=m\rho_j(a_j)\rho_i(a_i)$ for all $a_i, a_j \in A, m \in M, 1\leq i \neq j \leq n.$  Lastly, a \emph{left/right} $(l,r)$-\emph{multimodule} is an abelian group $M$ which has a left $l$-multimodule structure as well a right $r$-multimodule structure, where left and right actions commute in the sense that $(\iota_i(a_i)m)\rho_j(a_j)=\iota_i(a_i)(m\rho_j(a_j))$ for all $a_i, a_j \in A, m \in M, 1\leq i  \leq n, 1\leq j \leq n.$
\end{definition}
We will usually drop the $n$ or $(l,r)$ from our notation and simply call $M$ a multimodule.  

\begin{remark}
There are times when the type of action (left/right) will be unknown.  In these situations we let $\Lambda_i$ denote an action and we work to determine the type of action when the setting is clear.
\end{remark}
 
\begin{example}
An $A$-bimodule $M$ has a left action and a right action so $M$ is an example of an $A$ left/right (1,1)-multimodule with two actions.
\end{example}

We now provide the main Theorem of \cite{MR3338542} which tells us how many actions each simplicial set is allowed by identifying certain actions.  Note that when $A$ is commutative, any left action also satisfies the definition of a right action.  The module below is just assumed to be a multimodule.  
 
\begin{theorem} \cite[1.1]{MR3338542}
\label{higher}
Let $A$ be a commutative $k$-algebra.  Given a pointed simplicial set $X_{\bullet}$, there exists a cosimplicial $k$-module $(M,X)^{\bullet}$ associated to an $A$-module $M$ given by 

$$(M,X)^n= \hom_k (k \otimes_k 
\bigotimes\limits_{{\substack{\sigma \in X_n\\ \sigma \neq \ast}}} A,M)$$ 

with coface and codegeneracy maps given by

$$d_n^i f(1\otimes_k \bigotimes\limits_{{\substack{\sigma \in X_{n+1}\\ \sigma \neq \ast}}}a_{\sigma})=\prod\limits_{{\substack{\sigma \in X_{n+1}\\ d_i (\sigma)=\ast}}} (\Lambda_{(i,n)}^{\sigma}(a_{\sigma}))\cdot f(1\otimes_k \bigotimes\limits_{{\substack{\Omega\in X_n \\ \Omega \neq \ast}}} \prod\limits_{{\substack{\sigma \in X_{n+1} \\ d_i(\sigma)=\Omega}}}a_{\sigma})$$

and

$$s_n^i f(1\otimes_k \bigotimes\limits_{{\substack{\sigma \in X_{n+1}\\ \sigma\neq\ast}}}a_{\sigma})=f(1\otimes_k\bigotimes\limits_{{\substack{\Omega \in X_{n+1} \\ \Omega \neq\ast}}} 1\cdot \prod\limits_{{\substack{\sigma \in X_n\\ s_i(\sigma)=\Omega}}}a_{\sigma})$$

if the actions $\Lambda_{(-,-)}^-$ on $M$ satisfy the following for simplices $\sigma, \Omega$ and $\mu$:

\begin{itemize}
\item[i)]  $\Lambda_{(j,n+1)}^{\sigma} = \Lambda_{(i,n+1)}^{\sigma}$ if $\sigma \neq \ast , d_i (\sigma)=d_j(\sigma)=\ast$ and the dimension of $\sigma$ is at least 2 and $i<j.$  
\item[ii)]  $\Lambda_{(j,n+1)}^{\sigma}=\Lambda_{(j-1),n}^{\Omega}$ if $d_i(\sigma)=\Omega , d_j(\sigma)=\ast , d_{j-1}(\Omega)=\ast$ and the dimension of $\sigma$ is at least 2. 
\item[iii)]  $\Lambda_{(i,n)}^{\Omega}=\Lambda_{(j-1,n)}^{\mu}$ if $d_i(\Omega)=\ast , d_{j-1}(\mu)=\ast$ and there exists a $\sigma$ of dimension at least 2 where $d_j(\sigma)=\Omega$, $d_i(\sigma)=\mu$ and $i<j.$ 
\item[iv)]  $\Lambda_{(i,n)}^{\Omega}=\Lambda_{(i,n+1)}^{\sigma}$ if $d_i(\sigma)=\ast , d_i(\Omega)=\ast , d_j(\sigma)=\Omega$ and the dimension of $\sigma$ is at least 2.  
\end{itemize}
where $\Lambda_{(i,n)}^\sigma (a)$ represents the $\Lambda_{(i,n)}^\sigma$ action of $a \in A$ whenever $0\leq i\leq 1,$  $\sigma \in X_{n+1}$ and $d_i(\sigma)=\ast$.

\end{theorem}

\section{Not Necessarily Commutative Algebras}
\label{notnec}

We would like to develop higher order Hochschild (co)homology to accept non-commutative algebras, but in order to do so we need to provide an order in which to multiply elements in the equation from Theorem \ref{higher}.  In the following section we will determine algebras and modules allowed for higher order Hochschild cohomology and leave it to the reader to check that the same holds for homology. 

\begin{definition} Given a simplicial set $X_\bullet$ and a map $f\colon X_n\rightarrow X_m$, for an $f$-fiber $S=\{\sigma_i| f(\sigma_i)=\tau\}$ for some $\tau \in X_m$ we refer to an ordering on $S$ as an \emph{f-ordering} of $S$.  
\end{definition} 

Now if we consider all $d_i$-orderings on the corresponding subsets of $X_n$ we see that this gives a way to multiply elements of our algebra for each $d_i$ if $A$ is not necessarily commutative (where products of elements in each tensor factor are multiplied with elements represented by smaller simplices on the left and elements represented by larger simplices on the right).  The main issue is that in order for the associated cosimplicial $k$-module to satisfy the cosimplicial identities, we need the $d_i$-orderings to have some compatibilities.
We start with the following remark.
\begin{remark} For an two maps $f\colon X_n \rightarrow X_m$ and $g\colon X_m\rightarrow X_k$ we have two associated orders on $gf$-fibers.  We have the $gf$-ordering, but we also have \emph{composition induced} orderings.  In particular, if $\sigma_i$ and $\sigma_j$ are simplices in $X_n$ then $\sigma_i < \sigma_j$ in the composition induced ordering if and only if $\sigma_i < \sigma_j$ in the $f$-ordering or $f(\sigma_i)<f(\sigma_j)$ in the $g$-ordering otherwise $\sigma_j < \sigma_i$ in the composition induced ordering. 
\end{remark} 

\begin{definition}
Given a simplicial set $X_\bullet$ with a choice of simplicial orderings for each composition of face maps.  We say that $X_\bullet$ admits a \emph{not necessarily commutative multiplicative ordering} (which we denote NNCMO) if for each composition of face maps $f=d_{i_k}\cdots d_{i_1}\colon X_n\rightarrow X_{n-k}$ and $g=d_{i_{j+k}}\cdots d_{i_{k+1}}\colon X_{n-k}\rightarrow X_{n-j-k}$ the composition induced ordering from $gf$ agrees with the $gf$-ordering for all $gf$-fibers whose image under $gf$ is not the basepoint. 
\end{definition}
 We now have the following Theorem.
 
\begin{theorem}
Given a simplicial set $X_\bullet$, non-commutative algebra $A$ and $A$-multimodule (whose left actions can also be considered as right actions) $M,$ the Hochschild cosimplicial $k$-module $(A,M,X)^\bullet$ exists if and only if $X_\bullet$ admits an NNCMO.
\end{theorem}

\begin{proof}
From \cite{MR3338542} we see that if $A$ is a commutative algebra, then $(A,M,X)^\bullet$ exists.  To see that $(A,M,X)^\bullet$ still exists when $A$ is not necessarily commutative and $X_\bullet$ admits an NNCMO, we first notice that the  orderings of $X_\bullet$ give an ordering for the multiplication in each tensor factor for each coface map in $(A,M,X)^\bullet.$  It is straightforward to see that the compatibility of orderings, coming from the fact that $X_\bullet$ admits an NNCMO forces the composition of coface maps to be well defined with respect to the order in which elements in each tensor factor are to be multiplied.  For the other direction, notice if $(A,M,X)^\bullet$ exists, then each coface map provides an ordering for the associated face maps in $X_\bullet.$  We see that these orderings must be compatible for $(A,M,X)^\bullet$ to be a well defined cosimplicial $k$-module, which implies that $X_\bullet$ admits an NNCMO.    
\end{proof}

\begin{remark}
It should be noted that this construction is not functorial.  From the Theorem above, we see that higher order Hochschild cohomology with noncommutative algebras only works if the simplicial set that it is built over admits an NNCMO.  
\end{remark}

\subsection{Relationship to Pirashvili and Richter's construction}
\label{PR}
In \cite{MR1899698}, Pirashvili and Richter use a similar approach when defining functor homology.  Starting with the categories of noncommutative finite sets, $\mathcal{F}(as)$ and pointed noncommutative finite sets $\Gamma (as),$ they show that the traditional Hochschild chain complex is given by the composition of functors

$$L(A,M)\circ \hat{C}\colon \Delta^{op}\rightarrow \Gamma (as)\rightarrow \kmod$$
where $\mathcal{F}(as)$ and $\Gamma (as)$ are defined to be the categories of finite sets $\mathcal{F}$ and pointed finite sets $\Gamma$ with the property that maps have a total ordering on preimages and $\hat{C}$ is a lifting of the pointed simplicial circle $C\colon \Delta^{op}\rightarrow \Gamma.$

This construction is very similar to the way the current paper considers simplicial sets which admit an NNCMO.  In fact, any simplicial set $X_\bullet$ which has a lifting to $\hat{X}_\bullet \colon \Delta^{op}\rightarrow \Gamma (as)$ has the property that $(A,M,X)^\bullet$ is a cosimplicial $k$-module for any associative algebra $A.$  The main difference between the approach in \cite{MR1899698} and the approach here is that it is conceivable that $X_\bullet$ can admit an NNCMO without inducing $f$-orderings on the subsets of $X_n$ whose image is the basepoint.  This does not guarantee a lifting $\hat{X}_\bullet \colon \Delta^{op}\rightarrow \Gamma (as).$  We show that the absence of an $f$-ordering on such subsets requires $M$ to have a specific type of action in the next subsection.  In addition, the main goal of Pirashvili and Richter was to construct a generalization of Hochschild and cyclic homology, while the goal in the present paper is to provide a description of simplicial sets which still work with noncommutative algebras.  We determine an exact list of simplicial sets in Section \ref{SSTHATWORK} by using this approach. 

\begin{remark}
The category $\mathcal{F}(as)$ is isomorphic to the category $\Delta S$ of Fiedorowicz and Loday in \cite{MR998125}
\end{remark}

\subsection{Module actions}
Before getting to the main Theorem which determines precisely which simplicial sets allow noncommutative algebras, we will consider what multimodule actions arise if $A$ is noncommutative.  
For a commutative algebra $A$ and $A$-module $M,$ it can be seen that any left action is also by definition a right action (we will denote such an action as an \emph{lr action}), but for noncommutative algebras, coefficient modules need not have lr actions.  To see what actions on coefficient modules need not be lr actions we consider the following proposition.

\begin{proposition}
For the cosimplicial $k$-module $(A,M,X)^\bullet$ suppose there exists $\sigma < \tau \in X_n$ with $d_{i_k}\cdots d_{i_1}(\sigma)=d_{i_k}\cdots d_{i_1}(\tau)=\omega \neq \ast$ with $d_{i_r}\cdots d_{k+1} (\omega)=\ast$ and $d_{r-1}\cdots d_{k+1} (\omega)\neq \ast$ then $\Lambda_{(i_r,n-r)}^{d_{r-1}\cdots d_{k+1}(\omega)}$ is a 
\begin{itemize}
\item left action if there exists a set of maps $d_{j_r}\cdots d_{j_1}=d_{i_r} \cdots d_{i_1}$ with the property that $d_{j_{r-l}}\cdots d_{j_{1}}(\tau)=\ast$ but $d_{j_{r-l}}\cdots d_{j_{l}}(\sigma)\neq\ast$ for some $l<r.$
\item right action if there exists a set of maps $d_{j_r} \cdots d_{j_1}=d_{i_r}\cdots d_{i_1}$ with the property that $d_{j_{r-l}}\cdots d_{j_{l}}(\sigma)=\ast$ but $d_{j_{r-l}}\cdots d_{j_{l}}(\tau)\neq\ast$ for some $l<r.$ 
\item lr action if $\Lambda^{d_{j_{r-1}}\cdots d_{j_{k+1}}(\omega)}_{(i_r,n-r)}$ is both a right and left action.
\end{itemize}
\end{proposition}

\begin{proof}
Let $\sigma , \tau$ and $\omega$ be the simplicies from the proposition above.  We will prove the proposition for a left action, but note that an analogous proof works for right actions.  We can assume that there exists $1\leq t <s <r$ and $1\leq w <r$ so that 
$$d_{j_s}\cdots d_{j_1}(\sigma)=\ast$$
$$d_{j_{s-1}}\cdots d_{j_{1}}(\sigma)\neq \ast$$
$$d_{j_t}\cdots d_{j_1}(\tau)=\ast$$
$$d_{j_{t-1}}\cdots d_{j_{1}}(\tau)\neq \ast$$

and from \cite{MR3338542} $\Lambda^{d_{j_{s-1}}\cdots d_{j_{1}}(\sigma)}_{(j_s,n-s)}=\Lambda^{d_{j_{t-1}}\cdots d_{j_{1}}(\tau)}_{(j_t,n-t)}=\Lambda^{d_{j_{r-1}}\cdots d_{j_{k+1}}(\omega)}_{(i_r,n-r)}$ which we will simply denote as $\Lambda ,$ but $\delta^{j_1}\cdots \delta^{j_r}f(-)=\delta^{i_1}\cdots \delta^{i_r}f(-)$ and among other elements, we see that on the left we have $a_\tau$ acting of $f(-)$ followed by $a_\sigma$ acting on $a_\tau \cdot f(-)$ and on the right we have $a_\sigma a_\tau$ acting on $f(-)$ so we have $a_\sigma(a_\tau f(-) = (a_\sigma a_\tau )f(-)$ so $\Lambda$ is a left action.
\end{proof}

This gives whether an action $\Lambda$ is left/right or an lr action from the perspective of each simplex.  Recall from \cite{MR3338542} that there are many action identifications so to determine if $\Lambda$ need only be left/right we need to actually consider all action identifications i.e. if $\Lambda^\sigma =\Lambda^\tau$ and $\Lambda^\sigma$ is a left action while $\Lambda^\tau$ is a right action, then they are both the same lr action.

\begin{theorem}
Let $X_\bullet$ be a simplicial set with an NNCMO.  Let $A$ be a not necessarily commutative algebra and $M$ be an $A$-multimodule.  Let $C$ be the set of possible distinct actions by $X_\bullet$ (with whether the action is left or right or lr indicated) and $D$ be the set of actions of $A$ on $M.$  Let $f\colon C \rightarrow D$  be a map of sets which preserves action type, then $(A,M,X)^\bullet$ exists where each action $\Lambda^\sigma$ of $A$ on $M$ is determined by $f.$
\end{theorem}

The proof of this is a straight forward check that the cosimplicial identities hold so it is left to the reader. 

\section{Simplicial sets that work}
\label{SSTHATWORK}
In general, given a simplicial set $X_\bullet,$ it seems that it should be a daunting task to determine if $X_\bullet$ admits an NNCMO.  Unfortunately, it turns out that very few simplicial sets admit an NNCMO.  We start by showing which simplicial sets do not.  For a motivating example, we consider a simplicial model for $S^2$

\begin{proposition}
\label{SPHERE}
The minimal simplicial decomposition of $S^2$ with one nondegenerate $0$-simplex and one nondegenerate $2$-simplex does not admit an NNCMO.
\end{proposition}

\begin{proof}
Consider the $4$-simplices $[00112], [00122], [01122], [01112].$  Each of these simplices is carried to the $2$-simplex $[012]$ by the composition of face maps $d_2d_1=d_1d_3.$  We have a $d_1$-ordering for the sets $\{[00112],[01112]\}$ and $\{[00122],[01122]\}$, but there is also a $d_3$-ordering for the sets $\{[00112],[00122]\}$ and $\{[01122],[01112]\}.$  It can be seen that no ordering of these four simplicies for $d_2d_1=d_1d_3$ will agree with both the $d_3$-ordering and the $d_1$-ordering.
\end{proof}

One might guess that if $S^2$ does not admit an NNCMO, then it may not be possible for any simplicial set of dimension equal or greater than 2 to admit an NNCMO.  This brings us to the main Theorem.

\begin{theorem}
\label{MAIN}
Let $X_\bullet$ be a finite simplicial set, then $X_\bullet$ admits an NNCMO if and only if $X_\bullet$ is a one dimensional simplicial set, so $HH_{X_\bullet}^\ast(A,M)$ exists when $A$ is noncommutative if and only if the simplicial sets are one dimensional. 
\end{theorem}

Before proving Theorem \ref{MAIN}, let us consider examples of some one dimensional simplicial sets which do admit an NNCMO.  We start by describing a ``nicer" ordering on fibers of face maps. 

\begin{definition} 
We say that a pointed simplicial set $X_\bullet$ has a cyclic ordering if for every $n\geq 1,$ each set $X_n \smallsetminus \{\ast\}$ has an ordering with the property that if $\sigma , \tau \in X_n \smallsetminus \{\ast\}$ ($n\geq 2$) with $\sigma <\tau$ then $d_i(\sigma) \leq d_i(\tau)$ for all $0\leq i \leq n$
\end{definition}

\begin{remark}
It is an easy check to see that any set $X_\bullet$ which admits a cyclic ordering also admits an NNCMO, furthermore whether an action is left or right is also simple to see.  
\end{remark}

To motivate a proof that all one dimensional simplicial sets admit an NNCMO, we start with the following.

\begin{proposition}
Let $X_\bullet$ be the minimal simplicial decomposition of $\bigvee_{i\in I} S^1,$ then $X_\bullet$ has a cyclic ordering.
\end{proposition}

\begin{proof}
First notice that $X_\bullet$ has one $1$-simplex for each copy of $S^1$ and one degenerate $1$-simplex for the basepoint $\ast.$  The $1$-simplices of $X_1 \smallsetminus \{\ast\}$ can be ordered in any way, so all we need to do is order $X_n$ for larger $n.$  This can be done by giving a cyclic ordering for each sub-simplicial set $S_\bullet^1.$  One such ordering is 
$$S_1^1 : [01]$$
$$S_2^1 : [001]<[011]$$
$$S_3^1 : [0001]<[0011]<[0111]$$
$$S_4^1 : [00001]<[00011]<[00111]<[01111]$$

\end{proof}

We call the ordering above a cyclic ordering because we can actually order all of the $S^1$ $n$-simplices (including $\ast$ clockwise around a circle as is done in Figure \ref{fig:cyclic} (starting with $\ast =[0...0]=[1...1]$) and see that face maps $d_i$ have the property that $d_i(\sigma)=d_i(\tau)$ if and only if $\sigma$ and $\tau$ are in the $n+1-i$ and $n+2-i$ places around the circle.  From this, we can imagine that face maps essentially squeeze adjacent simplices together and do not change the order.

\begin{remark}
In the above ordering, $\Lambda^{[0...01]}$ is a right action, while $\Lambda^{[01...1]}$ is a left action.  This demonstrates a fact that we are already aware of--traditional Hochschild cohomology has the ability to work with non-commutative algebras and not necessarily symmetric bi-modules as coefficient modules.
\end{remark}
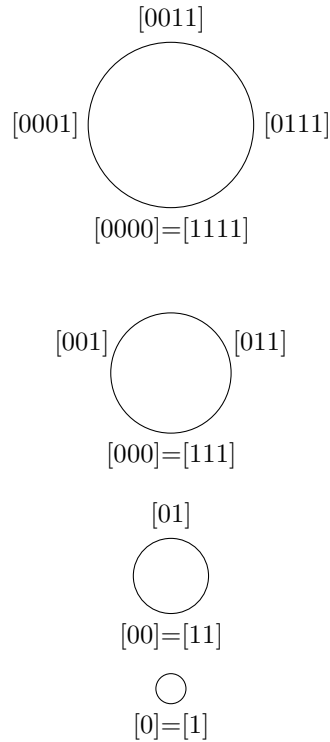
\begin{figure}
\centering

\begin{tikzpicture} 
\draw (8,0.2) circle [radius=0.2];
\draw (8,1.7) circle [radius=0.5];
\draw (8,4.4) circle [radius=0.8];
\draw (8,7.7) circle [radius=1.1];
\node at (8,0) [below] {[0]=[1]};
\node at (8,1.2) [below] {[00]=[11]};
\node at (8,2.2) [above] {[01]};
\node at (8,3.6) [below] {[000]=[111]};
\node at (7.3,4.8) [left] {[001]};
\node at (8.7,4.8) [right] {[011]};
\node at (8,6.6) [below] {[0000]=[1111]};
\node at (8,8.8) [above] {[0011]};
\node at (6.9,7.7) [left] {[0001]};
\node at (9.1,7.7) [right] {[0111]};

\end{tikzpicture}\caption{Cyclic ordering for $S_\bullet^1$}
  \label{fig:cyclic}
\end{figure}

We now proceed with the proof of Theorem \ref{MAIN}.

\begin{proof}[of Theorem \ref{MAIN}]
Suppose $X_\bullet$ is $n$-dimensional where $n\geq 2$ and let $\sigma$ be a nondegenerate simplex whose dimension is greater than or equal to 2.  Following a similar argument to the one in the proof of Proposition \ref{SPHERE} we see that the simplices $s_2s_0\sigma , s_3s_0 \sigma , s_3s_1 \sigma , s_1s_1 \sigma $ are carried to $\sigma$ via $d_2d_1 =d_1d_3.$  We have a $d_1$-ordering for the sets $\{s_2s_0\sigma , s_1s_1 \sigma\}$ and $\{s_3s_0\sigma , s_3s_1\sigma\}$ but we also have a $d_3$-ordering for the sets $\{s_2s_0\sigma ,s_3s_0\sigma \}$ and $\{s_3s_1\sigma ,s_1s_1\sigma\}.$  As in the proof of Proposition \ref{SPHERE} we see that this will not allow $X_\bullet$ to admit an NNCMO.  Now to see that any one dimensional simplicial set admits an NNCMO, we actually show that any one dimensional simplicial set has a cyclic ordering.  Notice that we can represent any simplex as the composition of degeneracy maps on a $1$-simplex.  Let us denote each simplex as follows: let $s_0s_3s_1s_1\sigma$ be denoted by $[001111]_\sigma$ if $\sigma$ is a nondegenerate $1$-simplex and $[000000]_\sigma$ otherwise.  By ordering the $1$-simplices, we can induce an order on the $n$-simplices by first ordering by subscript, so $[---\cdots-]_\sigma <[---\cdots-]_\tau$ if $\sigma <\tau$ in $X_1.$  We then order alphabetically.  For example, if $\sigma <\tau$ in $X_1$ then $[001]_\sigma <[011]_\sigma <[001]_\tau <[011]_\tau$ in $X_2.$  It is straight forward that this provides $X_\bullet$ with a cyclic ordering.   
\end{proof}
\section{Secondary Hochschild cohomology and pairs of simplicial sets}
\label{second}
In \cite{MR3465889} Staic introduced secondary Hochschild cohomology which was used to study $B$-algebra structures on $A[t]$ given $k$-algebras $A$ and $B$ with a map $\varepsilon \colon B \rightarrow A.$  In \cite{MR3566502} the author and Staic show that secondary Hochschild cohomology is a version of higher order Hochschild cohomology by generalizing Hochschild cohomology to pairs of simplicial sets $X_\bullet \subseteq Y_\bullet.$

To extend noncommutativity to pairs of algebras $A$ and $B$ we first consider the simplicial set $Y_\bullet.$  If $Y_\bullet$ is one dimensional, then $A$ and $B$ can both be non-commutative.  If $Y_\bullet$ is not one dimensional $B$ must be commutative.  In the second case, we then consider the simplicial set $X_\bullet.$  If $X_\bullet$ is one dimensional, then $A$ need not be commutative, however $\varepsilon (B)$ must be in the center of $A.$  For module coefficients, we simply consider where the action comes from in the simplicial set.

\section*{Acknowledgments}
I would like to thank Andrew Salch and Mihai Staic for conversations concerning this research.  I would also like to thank my family for their continued support; in particular, I am grateful to my loving wife Kendall, son Amos and daughter Flora.

%%%%%%%%%%%%%%%
%\section*{Acknowledgment}

%%%%%%%%%%%%%%%%%%%%%%%%%%%%%
%%%%%%%%%%%%%%%%%%%%%%%%%%%%%%%%%%%%%%%
%%%%%%%%

%%%%%%%%%%

\bibliography{Hochschild cohomology of noncommutative algebras}
\bibliographystyle{plain}

\end{document}